\documentclass[final,12pt]{article}
\usepackage{tikz}
\usetikzlibrary{arrows,decorations.pathmorphing,decorations.pathreplacing,backgrounds,positioning,fit}
\usepackage[conditional,light,first,bottomafter]{draftcopy}
\usepackage{graphics}
 \usepackage{graphicx}
 \usepackage[tight,footnotesize]{subfigure}
\usepackage{epsfig}
\usepackage{color}
\usepackage{amssymb,amsthm}
\usepackage{amsmath}
\usepackage{color}
\newtheorem{theorem}{Theorem}
\newtheorem{lemma}{Lemma}
\theoremstyle{definition}
\newtheorem{remark}{Remark}
\newtheorem*{acknowledgments}{Acknowledgments}

\def\De{\Delta}

\def\l{\lambda}

\def\s{\sigma}

\begin{document}
\begin{titlepage}
\title{Inverse problems for Jacobi operators I :\\ Interior mass-spring perturbations in finite systems
\footnotetext{%
Mathematics Subject Classification(2010):
47B36, 
15A29, 
47A75, 
15A18,
70J50.
}
\footnotetext{%
Keywords:
Finite mass-spring system.
Jacobi matrices.
Inverse spectral problem.
}}

\author{
\textbf{Rafael del Rio}
\\
\small Departamento de M\'{e}todos Matem\'{a}ticos y Num\'{e}ricos\\[-1.6mm]
\small Instituto de Investigaciones en Matem\'aticas Aplicadas y en Sistemas\\[-1.6mm]
\small Universidad Nacional Aut\'onoma de M\'exico\\[-1.6mm]
\small C.P. 04510, M\'exico D.F.\\[-1.6mm]
\small\texttt{delrio@leibniz.iimas.unam.mx}
\\[4mm]
\textbf{Mikhail Kudryavtsev}
\\
\small Mathematical Division\\[-1.6mm]
\small Institute for Low Temperature Physics and Engineering\\[-1.6mm]
\small Lenin Av. 47, 61103\\[-1.6mm]
\small Kharkov, Ukraine\\[-1.6mm]
\small\texttt{kudryavtsev@onet.com.ua}
\\[4mm]
}
\date{}
\maketitle
\vspace{-4mm}
\begin{center}
\begin{minipage}{5in}
  \centerline{{\bf Abstract}}
\bigskip

 We consider a linear finite spring mass system which is perturbed by modifying one mass
 and adding one spring. From  knowledge of the natural frequencies of the original and
 the perturbed systems we study when masses and springs can be reconstructed. This is
 a problem about rank two or rank three type perturbations of finite Jacobi matrices
 where we are able to  describe quite explicitly  the associated Green's functions.
 We give necessary and sufficient  conditions for  two given sets of points to be
 eigenvalues of the original and  modified system respectively.

\end{minipage}
\end{center}
\thispagestyle{empty}
\end{titlepage}


\section{Introduction}

We study a problem on inverse spectral analysis of Jacobi matrices.
Our motivation is to understand the behavior of oscillating
Micro Electro Mechanical Systems  (MEMS) which are often modeled
as  $N$ masses $m_0,m_1,...,m_{N-1}$
joined by $N+1$ springs with stiffness parameters (elastic coefficients)
$k_0,k_1,...,k_N$ and equilibrium lengths  $l_0,l_1,...,l_N$.

\begin {figure}[h]
\begin{center}
\begin{tikzpicture}
  [mass0/.style={rectangle,draw=black!60,fill=black!20,thick,inner sep=0pt,
   minimum size=5.7mm},
   mass1/.style={rectangle,draw=black!60,fill=black!20,thick,inner sep=0pt,
   minimum size=7.7mm},
   mass2/.style={rectangle,draw=black!60,fill=black!20,thick,inner sep=0pt,
   minimum size=7.0mm},
   dmass/.style={rectangle,draw=black!60,fill=black!20,thick,inner sep=0pt,
   minimum size=5mm},
   wall/.style={postaction={draw,decorate,decoration={border,angle=-45,
   amplitude=0.3cm,segment length=1.5mm}}}]
  \node (mass2) at (7.1,1) [mass2] {\footnotesize$m_2$};
  \node (mass1) at (4.25,1) [mass1] {\footnotesize$\,m_1$};
  \node (mass0) at (1.8,1) [mass0] {\footnotesize$m_0$};

\draw[decorate,decoration={coil,aspect=0.4,segment
  length=2.1mm,amplitude=1.8mm}] (0,1) -- node[below=4pt]
{\footnotesize$k_0$}  (mass0);
\draw[decorate,decoration={coil,aspect=0.4,segment
  length=1.5mm,amplitude=1.8mm}] (mass0) -- node[below=4pt]
{\footnotesize$k_1$} (mass1);
\draw[decorate,decoration={coil,aspect=0.4,segment
  length=2.5mm,amplitude=1.8mm}] (mass1) -- node[below=4pt]
{\footnotesize$k_2$} (mass2);
\draw[decorate,decoration={coil,aspect=0.4,segment
  length=2.1mm,amplitude=1.8mm}] (mass2) -- node[below=4pt]
{\footnotesize$k_3$} (9.3,1);
\draw[line width=.5pt,wall](0,1.5)--(0,0.7);
\draw[line width=.5pt,wall](9.38,0.7)--(9.38,1.5);

\end{tikzpicture}
\end{center}
\caption{mass-spring system with 3 masses}\label{fig:1}
\end{figure}
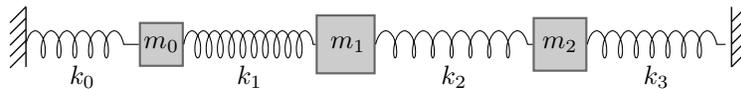

The first and last springs could be attached to fixed walls
as in Figure \ref{fig:1}, or free, in which case the model is the same just setting the elastic
coefficient of the spring at the free end  equal to zero. The masses are allowed to move in the
$x-$horizontal direction, with no friction and in absence of  external
forces. Using Hooke's and Newton's second laws, after normalization the following
equation is obtained:
\begin{equation*}
\frac{d^2}{dt^2}\vec{\nu}(t)=-J\vec{\nu}(t),
\end{equation*}
where the entry $\nu_i(t)$ of the $N$-dimensional vector $\vec{\nu}(t)$ describes
the position of mass $m_i$ at time $t$ (see \cite{Glad},\cite{Mar}) and $J$ is the Jacobi matrix
\begin{equation}
  \label{matrizJ}
 J= \begin{pmatrix}
    a_0 & b_0 & 0  &  0  &  \cdots
\\[1mm] b_0 & a_1 & b_1 & 0 & \cdots \\[1mm]  0  &  b_1  & a_2  &
\ddots &  \\
\vdots & & \ddots & \ddots &  \\
0 & 0 & \cdots & & a_{N-1}
  \end{pmatrix}
\end{equation}
with
\begin{equation}\label{klm}
 a_i=\frac{1}{m_i}(\gamma_{i+1}+\gamma_i) , \quad b_i=-\frac{\gamma_{i+1}}{\sqrt{m_i m_{i+1}}},
 \qquad \gamma_i=\frac{k_i}{l_i}
\end{equation}
(since we only use the fractions $\frac{k_i}{l_i}$, from now on we consider the elasticity parameters $\gamma_i=\frac{k_i}{l_i}$ of the springs instead of their length $l_i$ and the Hooke's coefficients $k_i$).
In the last years several experimental papers \cite{Har},\cite{Sple} were written on the possible methods
of determining micromasses with the help of oscillating microcantilevers, using the
spring-mass system approach. A possible theoretical basis for such a task is given by the paper
of Y.M.Ram \cite{Ram} (1993), who considered the inverse spectral problem of
reconstructing the Jacobi matrix \eqref{matrizJ} by its spectrum and the
spectrum of the perturbed matrix $\tilde J$ with
\begin{equation}
\tilde a_{N-1} = \frac{m_{N-1}}{\tilde m_{N-1}} a_{N-1} +\frac{ \gamma} {\tilde  m_{N-1}}, \quad
\tilde b_{N-2} = \sqrt{\frac{m_{N-1}}{\tilde m_{N-1}}} b_{N-2},
\qquad \tilde m_{N-1}>0, \ \gamma \in\mathbb{R},
\end{equation}
the other entries remaining without change. He obtained the necessary
and sufficient conditions for two point sets to be the spectra of such
a pair of matrices $J$, $\tilde J$ and provided a method of reconstructing
the matrices by the spectral data. Ram's results were partially extended
by P. Nylen and P. Uhlig in \cite{NU1},\cite{NU2} who considered the case of an analogous interior
perturbation affecting the entries $a_n, b_n, b_{n-1}$. Namely, they study
the problem of changing the mass $m_n$ by $\tilde m_n$ for a fixed $n$,
$0\leq n \leq  N-1$, and attaching  to it a new spring of elasticity parameter $k$,
 the other end of the spring being fixed:

\begin{figure}[h]
\begin{center}
\begin{tikzpicture}
  [mass1/.style={rectangle,draw=black!60,fill=black!20,thick,inner sep=0pt,
   minimum size=5.7mm},
   mass2/.style={rectangle,draw=black!60,fill=black!20,thick,inner sep=0pt,
   minimum size=7.7mm},
   mass3/.style={rectangle,draw=black!60,fill=black!20,thick,inner sep=0pt,
   minimum size=7.0mm},
   dmass/.style={rectangle,draw=black!60,fill=black!20,thick,inner sep=0pt,
   minimum size=5mm},
   wall/.style={postaction={draw,decorate,decoration={border,angle=-45,
   amplitude=0.3cm,segment length=1.5mm}}}]
  \node (mass3) at (7.1,1) [mass3] {\footnotesize$m_2$};
  \node (mass2) at (4.25,1) [mass2] {\footnotesize$\,m_1$};
  \node (mass1) at (1.8,1) [mass1] {\footnotesize$m_0$};
  \node (dmass) at (4.25,1.6) [dmass] {\scriptsize$\,\Delta m\,$};

\draw[decorate,decoration={coil,aspect=0.4,segment
  length=1.9mm,amplitude=1.8mm}] (3,1.6) -- node[above=4pt]
{\footnotesize$k$} (dmass);

\draw[decorate,decoration={coil,aspect=0.4,segment
  length=2.1mm,amplitude=1.8mm}] (0,1) -- node[below=4pt]
{\footnotesize$k_0$}  (mass1);
\draw[decorate,decoration={coil,aspect=0.4,segment
  length=1.5mm,amplitude=1.8mm}] (mass1) -- node[below=4pt]
{\footnotesize$k_1$} (mass2);
\draw[decorate,decoration={coil,aspect=0.4,segment
  length=2.5mm,amplitude=1.8mm}] (mass2) -- node[below=4pt]
{\footnotesize$k_2$} (mass3);
\draw[decorate,decoration={coil,aspect=0.4,segment
  length=2.1mm,amplitude=1.8mm}] (mass3) -- node[below=4pt]
{\footnotesize$k_3$} (9.3,1);
\draw[line width=.5pt,wall](0,1.5)--(0,0.7);
\draw[line width=.5pt,wall](3,2.1)--(3,1.3);
\draw[line width=.5pt,wall](9.38,0.7)--(9.38,1.5);

\end{tikzpicture}
\end{center}
\caption{Perturbed  mass-spring system}\label{fig:2}
\end{figure}
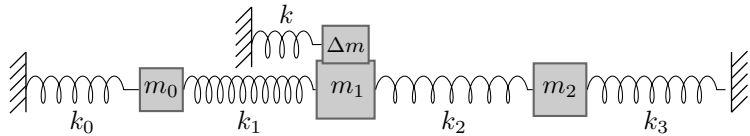

In this case the coefficients $a_n, b_{n-1}, b_{n}$ are modified as follows:
\begin{equation}\label{J pert}
\tilde a_n=\frac{1}{\tilde m_n}(\gamma_n
+ \gamma_{n+1}+\gamma)=\theta^2(a_n+M),
\quad \tilde b_{n-1}=\theta b_{n-1},
\quad \tilde b_n=\theta b_n,
\end{equation}
where
\begin{equation}\label{ga}
\theta=\theta(n):=\sqrt{\frac{m_n}{\tilde{m}_n}}, \quad M := \frac{\gamma}{m_n},\quad\gamma :=\frac{k}{l} 
\end{equation}

 \noindent$\tilde{m}_n$ denotes the perturbed mass and $l$ the  length of new spring. All other entries of $J$ remain unaltered.

The inverse problem for such a perturbation is the problem of reconstructing $J$
when we have the eigenvalues $\lambda_1<\lambda_2<\hdots<\lambda_N$
and $\tilde{\lambda}_1<\tilde{\lambda}_2<\hdots<\tilde{\lambda}_N$
of $J$  and $\tilde J$, respectively, where $\tilde J=\tilde J(n)$ is the
perturbed matrix with the changed coefficients \eqref{J pert} and all other entries the same
as $J$.\footnote{When $n=0$  we have only the two last equalities in (\ref{J pert})
since $b_{n-1}$ is not defined. Analogously if $n=N-1$ we have only the first and
last equalities in  (\ref{J pert}), since $b_n$ is not defined.}

P.Nylen and F.Uhlig obtained in \cite {NU1} necessary conditions for the spectra of the matrices
$J$ and $\tilde J$ and offered a method for reconstructing the possible
matrices in the cases where there is a finite number of solutions.
However, as it will be seen later on, their conditions are not sufficient
and the reconstruction does not give all the possible matrices.

The goal of this paper is to develop the direct and inverse spectral theory for the
interior perturbations of the described type \eqref{J pert}, to obtain the necessary
and sufficient conditions for the spectral data and the complete description
of the possible Jacobi matrices with such spectral data, providing an
explicit algorithm of reconstruction.

The main results are theorems \ref{CondNecSuf} (necessary and sufficient conditions)
and \ref{manifolds} (description of the possible matrices). The algorithm of reconstruction
is given in their proof. In forthcoming papers  we plan to give
the complete solution of this inverse problem for seminfinite matrices.

We shall use the following notations
\begin{equation}\label{DeltaM}
\Delta m_n = \tilde m_n-m_n,  \quad K := \frac{\gamma}{\De m_n}
\end{equation}
and assume that $\Delta m_n\geq 0$, i.e. the perturbed mass $\tilde m_n$ is greater than the initial mass $m_n$.

We shall have  the spectra of the initial and the perturbed Jacobi matrices
$\sigma (J)=\{\lambda _1,\lambda _2, ...,\lambda_N \}$,
$\sigma (\tilde J)=\{\tilde \lambda _1,\tilde \lambda _2, ...,\tilde\lambda_N \}$,
a given integer number $n\in\{0,\ldots,N-1\}$ which indicates the place the mass is modified and the parameter of
the perturbation $K=\frac{\gamma}{\Delta m_n}$ as the spectral data for the inverse problem of reconstructing $J$
and $\tilde J$. Notice that multiplying all the masses $m_i$ and elasticity
parameters $\gamma_i$ in \eqref{klm} by the same constant does not change the Jacobi matrix
\eqref{matrizJ} and the frequency characteristics of the free oscillations, so we obtain a
mechanically ``equivalent'' system. Thus, we cannot reconstruct exactly the masses and the elasticity
parameters from the frequencies. However, their fractions $\frac{m_i}{m_0}$, $\frac{\gamma_i}{\gamma_0}$,
$\frac{m_n}{\tilde m_n}=\theta^2$ , $\frac{\gamma}{m_n}=M$ will be found. That is the reason why we will
only work with fractional parameters of perturbation $\theta^2$, $M$, $K$ instead of the masses $m_i, \tilde m_n$
and elasticity parameters $\gamma_i, \gamma$.

Taking into account \eqref{ga} and \eqref{DeltaM}, we immediately have
\begin{equation}\label{k entre m'}
M = (\theta^{-2}-1) K,
\end{equation}
thus we always know one of the three parameters of perturbation $\theta^2$, $M$, $K$ from the two others.

We should remark that there is a huge variety of inverse spectral problems and many of them have been deeply
studied. We refer the reader to the books \cite{Chu-G} and \cite{Glad} for valuable
information on this important  subject.

\section{Preliminaries}

The eigenvector $\vec{c}$ of the matrix $J$, which satisfies
\begin{equation}\label{Jacobi}
J\vec{c}=\lambda\vec{c},
\end{equation}
can be written after normalization as
$\vec{c}=(P_0(\lambda),P_1(\lambda),P_2(\lambda),\hdots,P_{N-1}(\lambda))^T$
with $P_0(\lambda)=1$ where $P_i(\lambda)$ is a polynomial in $\lambda$ of
degree $i\in\{0,1,\hdots,N-1\}$.
From (\ref{Jacobi}) we get the expressions:

\begin{align}\label{polys}
a_0+b_0P_1(\lambda)&=\lambda \nonumber, \\
... \nonumber \\
b_{i-1}P_{i-1}(\lambda)+a_iP_i(\lambda)+b_iP_{i+1}&=\lambda P_i(\lambda),  \\
...  \nonumber \\
b_{N-2} P_{N-2}(\lambda)+a_{N-1}P_{N-1}(\lambda)&=\lambda P_{N-1}(\lambda). \nonumber
\end{align}

Therefore the polynomials $P_i$ are defined by the conditions
\begin{equation}\label{P}
P_{i+1}=\frac{1}{b_i}\{(\lambda-a_i)P_i-b_{i-1}P_{i-1}\},
\qquad P_{-1} =0, \qquad P_0=1,
\end{equation}
$i=0,...,N-2$.
It follows from the last equation in (\ref{polys}), that if $\lambda$ is an eigenvalue
then $Q_N(\lambda)=0$, where we define
\begin{equation}\label{Q}
Q_N(\lambda):=\lambda P_{N-1}-a_{N-1}P_{N-1}(\lambda)-b_{N-2}P_{N-2}(\lambda),
\end{equation}
and each root of the $N-$degree polynomial $Q_N$ is an eigenvalue of $J$, too. The polynomial
$Q_N$ is equal to  the characteristic polynomial of $J$ times a constant.

Let $P_i$ and $Q_N$ be defined as in (\ref{polys}) , (\ref{Q}) and let $\tilde P_i$, $\tilde Q_N$
be the corresponding polynomials for the perturbed operator $\tilde J$. We now try to get an
expression that relates $Q_N$ and $\tilde Q_N$.

Fix $n\in\{0,1,\hdots,N-1\}$
and  define the polynomials $\varphi_i$ for  $n\leq i\leq N $ as follows:
\begin{equation}\label{feo}
\varphi_n =0 \qquad \varphi_{n+1} =\frac{1}{b_n},
\qquad
\varphi_i =\frac {1}{b_{i-1}} \left\{(\lambda-a_{i-1})\varphi_{i-1}-b_{i-2}\varphi_{i-2} \right \},
\end{equation}
where we set $b_{N-1}=1$. Note that the degree of the polynomial $\varphi_i$ is $i-n-1$.
These are called polynomials of second kind while $P_i$ are called of
first kind, see \cite {Ak}.

\begin{lemma}\label{QtildeQ}
With the definition of $\varphi_i$ given above we have for any $n\in\{0,1,\hdots,N-1\}$
 \begin{equation}\label{cu}
 \tilde Q_N =\Gamma(n)\Big(Q_N +A\varphi _N P_n\Big)
 \end{equation}
where $A=\lambda (\theta^{-2}-1)-M $ and $\Gamma(n)=\theta(n)$
(see (\ref{ga})) if $n\in\{0,N-1\}$  and $\Gamma(n)=1 $ otherwise.
\end {lemma}

\begin{proof}
a) Case $0<n<N-1$. First we shall prove
\begin{equation}\label{efe}
\tilde P_{i} = \varphi_i A P_n +P_{i}
\end{equation}
for $n+1\leq i\leq N-1$.

Using (\ref {J pert}) and the recurrence relations (\ref{polys}) for $P_j$ and
$\tilde P_j$ we obtain
 \begin {align*} \tilde P_j &= P_j, \qquad\mbox{ if } \qquad j<n\\
 \tilde P_n&=\frac{1}{\theta}  P_n, \\
 \tilde P_{n+1}&=\frac{1}{\tilde b_n}\left\{(\lambda -\tilde a_n)
\tilde  P_n - \tilde b_{n-1} \tilde P_{n-1}\right\} \\
&=\frac {1}{\theta b_n}\left\{\left(\lambda -\theta ^2(a_n+ M)\right)\frac{ P_n }{\theta } -
\theta b_{n-1}P_{n-1}\right\}  \\
 &= \frac {1}{b_n}P_n \left ( \lambda (\frac {1}{\theta ^2} -1) -M\right )  +P_{n+1}. \\
\end {align*}
To obtain the last equation we added and subtracted
$ \lambda P_{n}- \frac{\lambda}{\theta^2 } P_{n}$. Therefore
\begin{align}
\tilde P_{n+1} &=\frac{1}{b_n} AP_n +P_{n+1} \label {Pi+1} \\
               &= \varphi _{n+1} AP_n +P_{n+1},
\end{align}
so (\ref{efe}) holds with $i=n+1$.

For $i=n+2$ we have
\begin{align*} \tilde P_{n+2}
&=\frac {1}{b_{n+1}}\left \{(\lambda -a_{n+1})\tilde P_{n+1}-b_{n}P_{n}\right\}\\
&=\varphi_{n+2}AP_n  +P_{n+2}
\end{align*}
so (\ref{efe}) holds in this case. Assume (\ref{efe}) holds for $i-2$ and $i-1$ where $i-2\geq n+1$
and let us prove it holds for $i$. First,
$$
\tilde P_{i}=\frac{1}{b_{i-1}}\left\{ (\lambda -a_{i-1})\tilde P_{i-1}-b_{i-2}\tilde P_{i-2}\right\}
$$
see (\ref{polys}), where the coefficients $a_{i-1},b_{i-1}$ and $b_{i-2}$ are unperturbed since $i\geq n+3$.
Therefore
\begin{align*}
\tilde P_i&=\frac{1}{b_{i-1}} \left\{
(\lambda -a_{i})(\varphi_{i-1}AP_n+P_{i-1})-b_{i-2}(\varphi_{i-2}AP_i+P_{i-2})\right\} \\
 &=\varphi_i AP_n +P_i
 \end{align*}
 and (\ref {efe}) holds. To prove (\ref{cu}) recall that
\begin{align*}
\tilde Q_N &=(\lambda -a_{N-1})\tilde P_{N-1}-b_{N-2}\tilde P_{N-2}\\
&=(\lambda - a_{N-1})( \varphi_{N-1}AP_n+P_{N-1})-b_{N-2}(  \varphi_{N-2}AP_n+P_{N-2}) \\
&=\varphi_NAP_n+Q_N
\end{align*}
and (\ref{cu}) is proven.

b) Case $n=0$. Note that $\tilde P_0=P_0 =1$. Using the first equality in (\ref{polys})
\begin{align*}
\tilde P_1= \frac{1}{\tilde b_o} (\lambda - \tilde a_0)
&=\frac{1}{\theta b_0}\left\{\lambda -\theta ^2(a_0+M)\right\} \\
   &=\theta \left\{P_1+ \varphi_1 A \right\}.
\end{align*}
Analogously we have
$$ \tilde P_2 =\theta \left\{P_2+ \varphi_2 A \right\}
$$
Using induction as in case a) we obtain $\tilde P_i =\theta \left\{P_i+ \varphi_i A \right\}$ for
$i=N-1 , N-2$ and  then $\tilde Q_N =\theta\Big(Q_N +A\varphi _NP_0\Big)$.

c) Case $n=N-1$. From (\ref{polys}) we know that
\begin{equation*}
P_{N-1}=\frac{1}{b_{N-2}} \lbrace(\lambda -a_{N-2})P_{N-2}  -b_{N-3}P_{N-3} \rbrace
\end{equation*}
and
\begin{equation*}
\tilde{P}_{N-1}=\frac{1}{\tilde b_{N-2}} \lbrace(\lambda
-a_{N-2})P_{N-2}  -b_{N-3}P_{N-3}\rbrace=\frac{b_{N-2}}{\tilde{b}_{N-2}}P_{N-1}=\frac{1}{\theta}P_{N-1}
\end{equation*}
Therefore
\begin{equation}
\tilde{P}_{N-1}=\frac{1}{\theta}P_{N-1}
\end{equation}

Analogously and considering  (\ref{Q}), we find that
\begin{align*}
\tilde{Q}_N(\lambda )&=(\lambda-\tilde{a}_{N-1})\tilde{P}_{N-1}
(\lambda )-\tilde{b}_{N-2}P_{N-2}(\lambda ) \\
  &=(\lambda-\theta^2 (a_{N-1}+M))
  \frac{1}{\theta}P_{N-1}-\theta b_{N-2}P_{N-2}
\end{align*}

Adding and substracting $\theta \lambda P_{N-1}- \frac{\lambda}{\theta } P_{N-1}$ we get
$$
=\theta Q_N(\lambda)+\theta  (\lambda(\frac{1}{\theta^2}-1)-M)P_{N-1}(\lambda)
$$

Therefore
$$
\tilde Q_N =\theta\Big(Q_N +A\varphi _NP_{N-1}\Big)
$$

Observe that $\varphi_N =1$ if $n=N-1$.
\end{proof}

Let us define the $jj$ Green's function
$$ G(z,j,j):= \langle \delta _j, (J-z)^{-1} \delta _j\rangle
$$
where $\langle.\rangle$ denotes scalar product and $\delta_j(i) = 1 $ if $i=j$ and
$0$ if $i\not=j $, $i,j\in\{0,1,\hdots,N-1\}$.

\begin{lemma}\label{Green} For  $0\leq n\leq N-1$
$$G(z,n,n)= -\frac{\varphi_N(z)P_n(z)}{Q_N(z)} $$
\end {lemma}
\begin{proof}
From the definition of polynomials $P_n$  see (\ref{polys}) or (\ref{P}) we know that,
for $n\geq 1$, $P_n(z)=0$ if and only if $z$ is an eigenvalue of $J_{[0,n-1]}$ the $n\times n$
upper left corner of $J$, that is:
\begin{equation*}
J_{[0,n-1]}=
 \begin{pmatrix}
  a_0    &       b_0                &             0            &        \cdots
   &
  \cdots         \\
  b_0    &       a_1                &           b_1         &       \cdots               &
   0         \\
   0        &       b_1                &           a_2         &        \cdots              &
   0         \\
                     \vdots         &               \vdots
                     &
                     \vdots          &                                 &        \vdots        \\
                             0         &                       0
                             &                                    0         &
                             \cdots               &     a_{n-1}       \\
\end{pmatrix}
\end{equation*}

Inductively from the definition it follows that
$$P_n(\lambda )=\frac{1}{b_0 b_1 \dots b_{n-1}}\lambda^n +\mbox{lower degree in } \lambda
$$
Since $b_0 b_1 \dots b_{n-1}P_n(\lambda )$ and $\det(\lambda -J_{[0,n-1]})$ are monic
polynomials
 of degree $n$ and with the same zeros, they are equal, i.e.
\begin{equation}\label{fea}
P_n(\lambda )=\frac{1}{b_o b_1 \dots b_{n-1}}\det(\lambda -J_{[0,n-1]})
\end {equation}

Analogously from (\ref{feo}) it follows inductively that
$$\varphi_N(\lambda)=\frac{1}{b_{N-2} \dots b_n}\lambda ^{N-n-1} +\mbox{lower degree in } \lambda
$$
(recall $b_{N-1}:=1$).
 If we define  the matrix , $n\leq N-2$,
 \begin{equation*}
J_{[n+1,N-1]}=
 \begin{pmatrix}
  a_{n+1}    &       b_{n+1}                &             0
  &        \cdots
  &          \cdots         \\
  b_{n+1}    &       a_{n+2}                &           b_{n+2}
  &       \cdots               &
  0         \\
   0        &       b_{n+2}                &           a_{n+3}
   &        \cdots              &
   0         \\
                     \vdots         &               \vdots
                     &
                     \vdots          &                                 &        \vdots        \\
                             0         &                       0
                             &                                    0         &
                              \cdots
                             &     a_{N-1}       \\
\end{pmatrix},
\end{equation*}
then $\varphi_i$ is related to the matrix $J_{[n+1,N-1]}$ in the same way as
$P_i$, defined by (\ref{polys}),
is related to the matrix $J$. In fact the $\varphi_i$
are the $P_i$ for the matrix $J_{[n+1,N-1]}$, multiplied by $\frac{1}{b_n}$.

We get similarly as (\ref{fea})
\begin{equation}\label{feaf}
\varphi_N(\lambda )=\frac{1}{b_n \dots b_{N-2}}\det(\lambda -J_{[n+1,N-1]})
\end {equation}
and
\begin{equation}\label{feaff}
Q_N(\lambda )=\frac{1}{b_0 \dots b_{N-2}}\det(\lambda -J)
\end {equation}
Therefore, if $1\leq n\leq N-2$,
\begin{equation}\label{feaffi0}
\frac{\varphi_N(\lambda ) P_n (\lambda )}{Q_N(\lambda )}=
\frac{\det(\lambda-J_{[0,n-1]})\det(\lambda - J_{[n+1,N-1]})}{\det(\lambda -J)} .
\end {equation}
When $n=0$
\begin{equation}\label{feaffi1}
\frac{\varphi_N(\lambda ) P_0 (\lambda )}{Q_N(\lambda )}=
\frac{\det(\lambda - J_{[1,N-1]})}{\det(\lambda -J)}
\end {equation}
and for $n=N-1$
\begin{equation}\label{feaffi2}
\frac{\varphi_N(\lambda ) P_{N-1} (\lambda )}{Q_N(\lambda )}=
\frac{\det(\lambda-J_{[0,N-2]})}{\det(\lambda -J)} .
\end {equation}
Recall $\varphi_N(\lambda )=1$ if $n=N-1$.

According to Cramer's rule the solution of $(z-J)u=\delta_n$ is the vector
$u=\left( u(0),u(1),...,u(N-1)\right)^T$ with
$$u(j)=\frac{\det(z-J)_j}{\det(z-J)}, \qquad j=0,...,N-1,
$$
where $(z-J)_j$ is
the matrix $z-J$ with the j-column substituted by $\delta_n$.
Since $u(n)=\langle \delta_n , (z-J)^{-1} \delta _n\rangle = -G(z,n,n)$, we get
\begin{equation}\label{feafff}
-G(z,n,n)=\frac{\det(z-J)_n}{\det(z-J)}
\end{equation}
Now observe $\det(z-J)_n =\det(\lambda-J_{[0,n-1]})\det(\lambda - J_{[n+1,N-1]})$ when $1\leq n\leq N-2$.
This can be checked
expanding the determinant on the left side by cofactors with respect to the $n$-th column.
When $n=0 , N-1$ we get $\det(z-J)_0 =\det(\lambda - J_{[1,N-1]})$ and
$\det(z-J)_{N-1} =\det(\lambda - J_{[0,N-2]})$, respectively.
Therefore from formulas (\ref{feaffi0}) to (\ref{feafff}) the
theorem follows.
\end{proof}

With the help of the two lemmas proved above, now we are ready to give an explicit formula for
the Green function at a point $nn$.

\begin{theorem}

\begin{equation}\label{GG}
G(\lambda,n,n) = \frac{1}{1-\theta^2} \Bigl\{
\frac{\theta^2}{\lambda-K}-\frac{1}{\lambda-K}
\Big(\frac{\prod_{j=1}^N (\lambda-\tilde{\lambda}_j)}{\prod_{j=1}^N (\lambda-{\lambda}_j)}\Big)\Bigr\},
\end{equation}

\end{theorem}
\begin{proof}

From lemmas \ref{QtildeQ}  and  \ref{Green}  we get
\begin{equation}\label{G}
\frac {\Gamma(n) Q_N-\tilde Q_N}{\Gamma(n) Q_N(\lambda (\theta^{-2} - 1) -
M)} =G(\lambda ,n,n)
\end{equation}

Similar to \eqref{feaff} we have
\begin{equation}\label{feafff2}
\tilde Q_N(\lambda )=\frac{1}{\tilde b_0 \dots
\tilde b_{N-2}}\det(\lambda -\tilde J)
\end {equation}
Using (\ref {J pert})
and writing the determinant as a product involving the eigenvalues of $\tilde J$ we get
\begin{equation}\label{ccu}
\tilde Q_N=\frac{\Gamma (n)}{\theta^2 b_0 \dots b_{N-2}}\prod_{j=1}^N (\lambda-\tilde{\lambda}_j)
\end {equation}
($\Gamma(n)$ defined in lemma \ref{QtildeQ}), and from(\ref{feaff})
\begin{equation}\label{cccu}
Q_N=\frac{1}{ b_0 \dots b_{N-2}}\prod_{j=1}^N (\lambda-{\lambda}_j)
\end{equation}

Using (\ref{ccu}) and (\ref{cccu}) in (\ref{G}) we get
 for  $0\leq n\leq N-1$
\begin{equation}\label{GG'}
\frac{m_n}{\lambda( \Delta m_n)-k}-\frac{\tilde m_n}{\lambda(\Delta m_n)-k}
\Big(\frac{\prod_{j=1}^N (\lambda-\tilde{\lambda}_j)}{\prod_{j=1}^N (\lambda-{\lambda}_j)}\Big)=G(\lambda,n,n),
\end{equation}
where we made use of (\ref{ga}). Multiplying both sides of the last equation
by $\frac{\De m_n}{\tilde m_n}=\theta^2-1$ and taking into account \eqref{DeltaM}
we get (\ref{GG}).

\end{proof}
\begin{remark} \label{re}
Multiplying both sides of (\ref {GG'}) by $\lambda(\Delta m_n)-k$ we obtain, if $\frac{k}{\Delta m_n}$ is not a pole of $G(\lambda ,n,n)$,
 \begin{equation}\label{GG1}
 m_n=\tilde m_n \frac{\prod_{j=1}^N (\lambda-\tilde{\lambda}_j)}{\prod_{j=1}^N
 (\lambda-{\lambda}_j)} \Longleftrightarrow   \lambda =\frac{k}{\Delta m_n}  \mbox{ or }
 G(\lambda ,n,n) =0
 \end{equation}
From  Lemma \ref{Green} and equations (\ref{fea}), (\ref{feaf}) we know that the roots of $G(\lambda ,n,n)$
are the eigenvalues of $J_{[0,n-1]}$ and $J_{[n+1,N-1]}$.It will be seen in Lemma \ref{LL},
that common eigenvalues of $J$ and $\tilde J$ are roots of   $G(\lambda ,n,n)$ too.
So if we know one of these points, or the value of $\frac{k}{\Delta m_n}$ if not a pole,
(for example not eigenvalue of $J$),  plus $m_n$ and the eigenvalues of $J$ and $\tilde J$,
then $\tilde m_n$ can be determined from (\ref{GG1}). A particular case of (\ref {GG1})
appears in \cite{Ram}, formula 2.17.
 \end{remark}
\section{Direct problem}

Let $\psi _1, \psi _2, ..., \psi _N$ be an orthonormal system of eigenvectors for $J$
with corresponding eigenvalues $\lambda _1, \lambda _2, ..., \lambda _N$. For example, one way
to get a system of orthonormal eigenvectors of $J$ is to consider
the polynomials $P_n$ defined by (\ref{P}) and normalize the eigenvectors
$(1,P_1(\lambda _k),P_2(\lambda _k), ..., P_{N-1}(\lambda _k))$  corresponding to  eigenvalue $\lambda _k$.
We obtain in this case
$|\psi_k(i)|^2=\frac{|P_i(\lambda _k)|^2}{\sum_{l=0}^{N-1}|P_l(\lambda _k)|^2}$.

\begin{lemma}\label{Green1}
$$
G(z,j,j):= \langle \delta _j, (J-z)^{-1} \delta _j\rangle =\sum_{k=1}^N \frac{|\psi _k(j)|^2}{\lambda _k-z}
$$
\end{lemma}
\begin{proof}
Let $\delta _j = \sum_{k=1}^N\alpha _k\psi _k$, then
$\alpha _k=\langle \delta _j,\psi _k\rangle=\psi _k(j)$
and
\begin {align*}
\langle \sum_{k=1}^N\alpha _k\psi _k,
(J-z)^{-1}\sum_{l=1}^N\alpha _l\psi _l\rangle&=\sum_{k,l=1}^N\overline{\alpha _k}\alpha _l\langle\psi _k,
\frac{1}{\lambda _l-z} \psi _l\rangle    &= \sum_{k=1}^N  \frac{|\psi _k(j)|^2}{\lambda _k-z}
\end{align*}
\end{proof}

The following Lemma can be proven using Lemma \ref{Green1}. The spectrum of the operator $T$ will
be denoted by $\sigma (T)$.
\begin  {lemma} \label{L}
Let $\beta \in \sigma(J)=\{\lambda _1,\lambda _2, ...,\lambda _n\} $\\

a) If $0<n<N-1$, then $G(\beta ,n,n)=\infty$ or $G(\beta ,n,n)=0$ \\

b) If $n=0$ or $n=N-1$, then $G(\beta ,n,n)=\infty$.

\end{lemma}

\begin{proof}
 a) From Lemma \ref{Green1}, we know

$$G(z,n,n) =\sum_{k=1}^N \frac{|\psi _k(n)|^2}{\lambda _k-z}    $$

Assume $\beta =\lambda _{k_0} \in \sigma (J)$. There are two possibilities:
either $\psi_{k_0}(n)\not=0$ or $\psi _{k_0}(n)=0$. If the
first holds, then $\lambda _{k_0}$ is a pole of $G(z,n,n)$. In fact, taking
the left and right limits along the real axis we get $G(\lambda_{k_0}-,n,n)=+ \infty$
and $ G(\lambda_{k_0}+,n,n)=- \infty$ respectively. If the second possibility holds,
then the eigenvector $\psi _{k_0}$, which corresponds to
the eigenvalue of $\beta =\lambda _{k_0}$ of $J$, vanishes at $n$. This implies that
$\lambda_{k_0}$ is eigenvalue of $J_{[0,n-1]}$ and $J_{[n+1,N-1]}$ too.
Therefore $\varphi _N(\lambda _{k_0})=P_n(\lambda _{k_0})=0$ and using Theorem \ref{Green}
follows that $G(\beta = \lambda _{k_0},n,n)=0$.

b) Observe that $P_0 =1$ and similarly $P_{N-1} \not=0$, since otherwise
 $Q_N(\lambda _{k_0}) =P_{N-1}(\lambda _{k_0}) =0 $ and this
would imply $P_i = 0$ for all $i$.
 Therefore, using $|\psi_k(r)|^2=\frac{|P_r(\lambda _k)|^2}{\sum_{l=0}^{N-1}|P_l(\lambda _k)|^2}$,
$\psi _{k_0}(r)\not=0$ if $r= 1$ or $N-1$
 and then  $G(\beta ,n,n)=\infty$ follows.
 \end{proof}

Using Lemma \ref{Green1} and formula (\ref {GG}) we obtain the next formula, which will
be used in the following results:
\begin{equation} \label{Funda}
\frac{\theta^2}{1-\theta^2} - \frac{1}{1-\theta ^2}N(\lambda )
 =(\lambda -K)\sum_{l=1}^N \frac{|\psi _l(n)|^2}{\lambda _l-\lambda},
\end{equation}
where
 \begin{equation} \label{Fun}
  N(\lambda ):= \frac{\prod_{j=1}^N (\lambda-\tilde{\lambda}_j)}{\prod_{j=1}^N (\lambda-{\lambda}_j)}
\end{equation}
 and  $\psi _1, \psi _2, ..., \psi _N$ is an orthonormal system of eigenvectors
 of J with corresponding eigenvalues $\lambda _1, \lambda _2, ..., \lambda _N$.

 \medskip

 We call an eigenvalue $\lambda_r\in\sigma (J)$ unmovable if it does not change after the perturbation,
 i.e. if $\lambda_r\in \sigma (J)\cap \sigma (\tilde J)$.

\begin  {lemma} \label{LL} {\rm (Properties for the unmovable eigenvalues and $\theta$)}
\begin{equation}\label{coin}
\sigma (J)\cap \sigma (\tilde J)= \sigma (J)\cap \Big (\{\lambda :G(\lambda ,n,n)=0\}\cup
\{K\}\Big)
\end{equation}
\begin{equation}\label{coincide}
\# \Bigl\{ \sigma (J)\cap \sigma (\tilde J) \cap \{\lambda :G(\lambda ,n,n)=0\} \Bigr\} \leq \min(n,N-n-1).
\end{equation}
\begin{equation}\label{ora}
 \lambda_r \in \sigma (J)\cap\{\lambda :G(\lambda ,n,n)=0\}\Longrightarrow N(\lambda_r )
 =\theta^2\end{equation}
\begin{equation}\label{ora1}
\lambda_r =K\in\sigma (J)\Longrightarrow  N(\lambda_r )
=\theta^2+(1-\theta ^{2})|\psi_r(n) |^2\geq\theta^2
\end{equation}
\begin{equation}\label{ora2}
K\not\in\sigma (J)\Longrightarrow  N(K)=\theta^2
\end{equation}
\begin{equation}\label{o!}
 K\in  \s(J) \cup \s(\tilde J)\Longrightarrow K \in \s(J)\cap  \s(\tilde J).
\end{equation}
If $\l_N\ne K$, then $\tilde\l_N\ne\l_N$
and if $\l_1\ne K$, then $\tilde\l_1\ne\l_1$.
\end{lemma}
\begin{proof}
"$\supset$" for \eqref{coin}. \
 If $\lambda \in \sigma (J)$ but  $\lambda \notin \sigma (\tilde J)$ then
 $\lambda $ is a pole of $N$, by the definition of $N$.

If  $G(\lambda ,n,n) =0$, then from (\ref{GG}),(\ref {Funda}), we get
\begin {equation} \label{GGG}
\theta^2=N(\lambda )
\end{equation}
and $\lambda$ is not a pole of $N$. Therefore, if one eigenvalue $\lambda_{j_0}$ of $J$
coincides with one root of $G(\lambda,n,n)$, then some eigenvalue of $\tilde J$ has to
coincide with $\lambda _{j_0}$. Since the spectra of $J$ and $\tilde J$ are simple, no
more eigenvalues coincide at that point. In case  the  eigenvalue $\lambda _{j_0}$
of $J$ coincides with $K$, then from  \eqref{Funda} we get
\begin{equation}
 \frac{\theta^2}{1-\theta^2} - \frac{1}{1-\theta ^2}N(\lambda ) \longrightarrow
  -|\psi _{j _0}|^2 \;\mbox{when}\;  \lambda \longrightarrow K
\end{equation}
 and $\lambda_{j_0}=K$ is not a pole of $N$. Then some
 eigenvalue of $\tilde J$ has to coincide with $\lambda_{j_0}$.

 "$\subset$" for \eqref{coin}. \
 If $\lambda \in \sigma (J)\cap \sigma (\tilde J)$ then $\lambda$
 is not a pole of $N$, by the definition of $N$. It is enough to consider
the case $\lambda \not=K$. Using Lemma \ref{L}  we know that $\lambda $
is either a pole or a zero of $G(\cdot,n,n)$. From (\ref {Funda}), poles of $G$
that are not $K$, are poles of $N$. Therefore $\lambda$ is
a root of $G$ if $\lambda \in \sigma (J)\cap \sigma (\tilde J)$ and $\lambda \not=K$.

Further, as mentioned in Lemma \ref{L} a), if an
eigenvalue of $J$ is zero of $G(\lambda,n,n)$, then it is a common eigenvalue of $J_{[0,n-1]}$
and $J_{[n+1,N-1]}$. Since, according to \cite{MR1616422}, there are at most
$\min(n, N-1-n)$ of them, we conclude that the common eigenvalues of
$J$ and $\tilde J$ are at most $\min(n,N-1-n)$ plus possibly the point $K$, so
\eqref{coincide} is proven. Notice, that by b) of Lemma \ref{L},
$\sigma (J) \cap \{\lambda :G(\lambda ,n,n)=0\}=\emptyset$ when $n=0$ or $N-1$,
so the only possible common eigenvalue for
$J$ and $\tilde J$ in this case is  $K$.

Implications \eqref{ora}--\eqref{o!} are obtained using the proved part of the Lemma
and formula (\ref{Funda}).

Let us now pass to the last assertion. If the matrix $A_{[0,N-2]}$ is constructed from
a Jacobi matrix $A_{[0,N-1]}$ by
 deleting the last column and row ( similarly first column and row), then
 $\sigma (A_{[0,N-2]})\subset (\beta _1,\beta  _N)$
where $\beta  _1,\beta  _N$ are, respectively, the smallest and  largest eigenvalues
of $A_{[0,N-1]}$. See Corollary 2.5 in \cite{MR1616422}, for example.
Using this fact with matrices $J=J_{[0,N-1]}, J_{[0,N-2]},...$ and so on, we get
that the eigenvalues $\lambda _1$ and $\lambda _N$ of $J$, cannot be eigenvalues
of the submatrices $J_{[0,n-1]}$, $J_{[n+1,N-1]}$  defined in Theorem \ref{Green},
and therefore  cannot be roots of $G(\cdot,n,n)$. From Lemma \ref{LL} follows
$\lambda _r\not=\tilde\lambda _r$, $r=1,N$, unless  equal to $K$.
\end{proof}

\begin{theorem}\label{Condnes} {\rm (Interlacing properties for two spectra)}
Let $\sigma (J)=\{\lambda _1,\lambda _2, ...,\lambda_N \}$,
$\sigma (\tilde J)=\{\tilde \lambda _1,\tilde \lambda _2, ...,\tilde\lambda_N \}$.
Take   $p\in\{ 0, 1, ..., N\}$ such that  $\lambda_p<K\leq \lambda _{p+1}$,
where we define $\lambda _0=-\infty$ and $\lambda _{N+1}=\infty$, being $K$ the parameter
of the perturbation \eqref{J pert}, \eqref{DeltaM}.
Then there is exactly one eigenvalue of $\tilde J$ in each of the following
intervals:
$$[\lambda _{j}, \lambda _{j +1}),\quad j=1,...,p-1,$$
 $$(\lambda _{j},\lambda _{j+1}],\quad j=p+1,...,N-1,$$
 and
 $$[\lambda _p, K),\qquad
 (K,\lambda _{p+1}]\quad\mbox{if this last interval not empty.}$$

 \end{theorem}

\begin{proof}
Let us find if there is  an eigenvalue of $\tilde J$, that is a root of $N(\lambda)$, in the interval
$[\lambda_p, K)$. Using (\ref {Funda}) and (\ref{ora1}) we get
$N(K) \geq  \frac {m_n}{\tilde m_n}>0$. From Lemma \ref {L} we have
two possibilities at  $\lambda_p$ . Either $G(\lambda_p,n,n) =0$ or $\infty$.
If the first happens, then $\lambda_p$ coincides with an eigenvalue of $\tilde J$
by Lemma \ref {LL} (and $N(\lambda_p) =\theta$ by (\ref{ora})).
   If the second possibility holds, then from (\ref {Funda})
  $N(\lambda_p+)=-\infty$ .
Since $N$ is continuous in  $(\lambda_p, K)$ there is at least
a zero of $N$ in this interval.  Therefore there is at least one eigenvalue of
$\tilde J$ in the interval $[\lambda_p, K)$.

Now consider the interval $[\lambda_{p-1}, \lambda_p)$.  The two options mentioned above
for $\lambda_p$ give us either $N(\lambda_p)=\theta$ or $N(\lambda_p-) = \infty$.
The two options for $\lambda_{p-1}$ are $N(\lambda_{p-1})=\theta$ and in this case
$\lambda_{p-1}$ is eigenvalue of $\tilde J$, or $N(\lambda_{p-1}+)=-\infty$. If the second possibility
happens, since $N(\lambda )>0$ for
$\lambda$  near $\lambda_p$,  again from the continuity of  $N(\lambda)$  in $(\lambda_{p-1}, \lambda _p)$
follows that $N$ has at least a zero in this interval. Therefore, there is at least one eigenvalue o
f $\tilde J$ in $[\lambda_{p-1}, \lambda_p)$.
 Continuing in this way, we get one eigenvalue of $\tilde J$ in each interval
 $[\lambda _{j}, \lambda _{j +1})$
 with $j=1,...,p-1$. Therefore if
$K > \lambda _p$ we get at least one eigenvalue of
$\tilde J$ in each $[\lambda _{j}, \lambda _{j +1})$ with $j=1,...,p-1$
and one  in $[\lambda_p, K)$, a total of at least $p$ eigenvalues of
$\tilde J$ in the interval $[\lambda_1, K)$.

Now consider the case $K <  \lambda_{p+1}$ and let us see whether
there is an eigenvalue of $\tilde J$ in $(K, \lambda_{p+1}]$ .
From (\ref {Funda}) and (\ref{ora1}) , $N(K) \geq \theta$.
Now, from Lemma \ref {L} we have two possibilities at  $\lambda_{p+1}$.
Either $G(\lambda_{p+1},n,n) =0$ or $\infty$. If the first happens,
then $\lambda_{p+1}$ coincides with an eigenvalue of $\tilde J$
by (\ref{coin}) (and $N(\lambda_{p+1}) =\theta$ by (\ref{ora})).
 If the second possibility holds, then from (\ref {Funda}),
$N(\lambda_{p+1}-)=-\infty$ .
Since $N$ is continuous in  $( K,\lambda _{p+1})$ there is at least
a zero of $N$ in this interval.  Therefore there is at least one eigenvalue of $\tilde J$
in the interval $( K,\lambda_{p+1}]$. Now consider the interval
$(\lambda_{p+1},\lambda_{p+2}]$.  The two options mentioned above for $\lambda_{p+1}$
give us either $N(\lambda_{p+1})=\theta$ or $N(\lambda_{p+1}+) = \infty$.
The two options for $\lambda_{p+2}$  are $N(\lambda_{p+2})=\theta$
and in this case $\lambda_{p+2}$ is eigenvalue of $\tilde J$, or $N(\lambda_{p+2}-)=-\infty$.
If the second possibility happens, since $N(\lambda )>0$ for
$\lambda$  near $\lambda_{p+1}$,  again from the continuity of  $N(\lambda )$
in $(\lambda_{p+1}, \lambda_{p+2})$ follows that $N$ has at least a zero in this interval.
Therefore, there is at least one eigenvalue of $\tilde J$ in $(\lambda_{p+1}, \lambda_{p+2}]$.
Continuing in this way, we get one eigenvalue of $\tilde J$  in each interval $ (\lambda _{j},\lambda _{j+1}]$
with $j=p+1,\ldots,N-1$.

Therefore if $K < \lambda_{p+1}$ we get at least one eigenvalue of $\tilde J$
in each $(\lambda _{j}, \lambda _{j+1}]$ with $j=p+1,\ldots,N-1$
and one  in $ (K,\lambda_{p+1}]$, a total of at least  $N-p$ eigenvalues of $\tilde J$
in the interval $ (K,\lambda _N]$. Since the $p$ eigenvalues of
$\tilde J$ in $[\lambda _1, K)$ plus the $N-p$ eigenvalues of $\tilde J$ in
$ (K,\lambda _N]$ give all the eigenvalues of $\tilde J$, we conclude that at most
there is one eigenvalue of the perturbed operator in each one of the intervals considered.

In case $K = \lambda_{p+1}$ we analyze first the interval
$(K,\lambda_{p+2}]$ exactly as above, and find at least one eigenvalue of
 $\tilde J$ in it. Continuing with the other intervals as before, we conclude that there is at least
 one eigenvalue of  $\tilde J$ in $ (\lambda_{j},\lambda_{j+1}]$ with $j=p+2,...,N-1$.
Therefore we get $N-p-1$ eigenvalues of $\tilde J$ in $ (K,\lambda_N]$.
These plus $\lambda_{p+1}$ and the $p$ eigenvalues in $[\lambda_1, K)$
give all the $N$ eigenvalues of the perturbed operator. Therefore there is {\it at most} one
eigenvalue of $\tilde J$ in each of the intervals considered.
\end{proof}

 \begin  {lemma} \label{Alternativa}
If $K\in\s(J) \cap \s(\tilde J)$, then the following alternative holds:

{\sl either}

(a) $N(K) = \theta^2$, and then $N'(K)=0$, $G(K,n,n)=0$ and there are
at most $\min(n-1,N-n-2)$ other common points of $\s(J)$
and $\s(\tilde J)$.

{\sl or}

(b) $N(K) > \theta^2$, and then $G(K,n,n)=\infty$ and there may be
$\min(n,N-n-1)$ other common points of $\s(J)$
and $\s(\tilde J)$.

 \end{lemma}
 \begin{proof}
According to Lemma \ref{L}, $G(K,n,n)=0$ or $G(K,n,n)=\infty$.
In the first case, by formula \eqref{GG}, that means that the function
$$ \frac{1}{\l-K} \cdot \frac{\theta^2 - N(\lambda)}{1-\theta^2}
$$
has a zero at $K$, thus
$$ N(K) = \frac {\displaystyle\prod_{\tilde\l_j\ne K}
(K-\tilde\l_j) } {\displaystyle\prod_{\l_j\ne K} (K-\l_j) }
= \theta^2
$$
and, moreover,
$$ N'(K)=0.
$$
Since $G(\l,n,n)$ may vanish only at $\min(n,N-n-1)$ points of the spectrum of $J$, in the first case there
may be at most $\min(n-1,N-n-2)$ other points of $\s(J)\cap\s(\tilde J)$.

In the second case, the function $\theta^2 - N(\lambda)$ may not
have a zero at $K$ of order greater than $1$, because otherwise $G(\l,n,n)$ would not have
a pole at this point. We also know that $G(\l,n,n)$ has a negative residue at $K$.
Thus, $N(K)>\theta^2$. In the second case there may be $\min(n,N-n-1)$
more points of $\s(J)\cap\s(\tilde J)$ where $G(\l,n,n)$ vanish.
\end{proof}

\section{Inverse problem}

It turns out that the properties of the spectral data, described in the previous section, are sufficient.

Let be given:

i)
 $\s=\{\l_1,\ldots,\l_N\}$ and $\hat\s=\{\hat\l_1,\ldots,\hat\l_N\}$ two finite  subsets of $\mathbb{R}$
$\l_i<\l_{i+1},\quad \hat\l_i<\hat\l_{i+1}$ for $i=1,\ldots,N-1$.

ii) $K\in\mathbb{R}$

iii) An integer number $n\in\{0,1,\ldots,N-1\}$\\

 Introduce the following notations:
\begin{equation}\label{mus}
\{\mu _1,\ldots,\mu _q\}:=\sigma \cap \hat\sigma -\{K\}.\ \mbox{If}\ \sigma \cap \hat\sigma -\{K\}=\emptyset, \ \mbox{then} \ q:=0\,;
\end{equation}
\begin{equation} \label{muss}
  \hat N(\lambda ):= \frac{\prod_{j=1}^N (\lambda-\hat{\lambda}_j)}{\prod_{j=1}^N (\lambda-{\lambda}_j)}\,;
\end{equation}

\begin{equation} \label{musss}
\theta ^2:=  \hat N(\mu _1). \ \mbox{If}  \ \sigma \cap \hat\sigma -\{K\}=\emptyset,
\ \mbox{then fix any } p \in (0,\hat N(K )]
\ \mbox{and set} \ \theta ^2 :=p.
\end{equation}

\begin{equation} \label{ñ}
   \tilde n:=\min(n,N-n-1)
\end{equation}

\begin{theorem}\label{CondNecSuf} {\rm (Necessary and sufficient conditions)}\\

The conditions \\

I)  $\sigma ,\hat\sigma$ and $K$ interlace as in Theorem \ref{Condnes}.\\

II)   $\theta ^2=  \hat N(\mu _1)=\hat N(\mu _2)=... =\hat N(\mu _q)\in(0,1)$\\

III) If $K\notin\sigma \cup \hat \sigma$ then $q\leq \tilde n$ and  $\hat N(K)=\theta ^2$.\\

IV) If $K\in\sigma \cup \hat \sigma$ then $K\in\sigma \cap \hat \sigma$ and either\\
   a) $q\leq \tilde n$ and $\hat N(K)> \theta ^2$
    or \\
   b) $q<\tilde n$  and  $\hat N(K)=\theta ^2$, $\hat N'(K)=0$  ( $'$ denotes derivative) \\

\noindent are necessary and sufficient for the existence of $N\times N$ Jacobi matrices $J$ and $\tilde J$,
where $\tilde J$ is obtained by perturbing $J$ at the $n$ place as described in  (\ref{J pert}),
that is $\tilde a_n
=\theta^2(a_n+M),
\quad \tilde b_{n-1}=\theta b_{n-1},
\quad \tilde b_n=\theta b_n$ with $M = (\theta^{-2}-1) K$, such that $$\sigma =\mbox{spectrum of}\quad J,\quad \hat\sigma =\mbox{spectrum of}\quad\tilde J.$$

\end{theorem}
\begin{remark}
Observe that condition I) implies $\hat N (K)\in (0,1)$
\end{remark}
 \begin{theorem}\label{manifolds}

 Assume conditions of previous theorem hold. If $K\notin\sigma \cup \hat \sigma$ or $K\in\sigma \cup \hat \sigma$
 and  option IV a) happens, then  there are infinitely many pairs $J, \tilde J$ of $N\times N$ Jacobi Matrices,
 if $q\not=0$. Indeed this inverse spectral family is a collection of
 $${N-2q-1 \choose n-q}$$ disjoint manifolds of dimension $q$ and diffeomorphic to a $q$ dimensional open ball.
 If $q=0$, then there is only the finite
 $${N-1 \choose n}$$ number of pairs $J, \tilde J$.\\
 If $K\in\sigma \cup \hat \sigma$ and option  IV)b happens, there are infinitely many pairs. The inverse
 spectral family is a collection of  $${N-2q-3 \choose n-q-1}$$ disjoint manifolds of dimension $q+1$
 diffeomorphic to a $q+1$ dimensional open ball.

\end{theorem}

\begin{remark} Note that when $K\notin\sigma \cup \hat \sigma$ and in case IVb) the parameter $\theta^2$
is uniquely determined by the spectral parameters i--iii), and in case IVa) $\theta^2$ is arbitrary in
$ (0,\hat N(K ))$. In particular, in Theorem \ref{manifolds} this means that if
$\sigma \cup \hat \sigma = \{K\}$ ($q=0$) and $\hat N'(K)=0$, then there are a finite number ${N-1 \choose n}$
of solutions for each $\theta^2\in(0,\hat N(K ))$ {\sl and} a collection of ${N-3 \choose n-1}$
disjoint one-dimensional manifolds of solutions for $\theta^2=\hat N(K)$.
\end{remark}

\begin{proof}

Here we  prove simultaneously Theorems \ref{CondNecSuf} and \ref{manifolds}.

The necessity of the conditions I)--IV) is already proved in previous section: Theorem \ref{Condnes}
proves necessity of  condition I). Assertions \eqref{coin} and \eqref{ora} of Lemma \ref{LL} prove
necessity of
condition II). For condition III) use assertions \eqref{ora2} and \eqref{coincide} of  Lemma \ref{LL}.
First part of condition IV) is  \eqref{o!}. Lemma \ref{Alternativa} a)
implies condition IV)b. Lemma \ref{Alternativa} b) and \eqref{ora1}, imply IV) a.
Now we prove the sufficiency part of theorem \ref{CondNecSuf}
and theorem \ref{manifolds} by finding all pairs of Jacobi matrices that have the given spectral data.

Consider the function
\begin{equation}\label{hat G}
\hat G(\l) = \frac{1}{1-\theta^2} \Bigl\{ \frac{\theta^2}{\l-K} - \frac{\hat N(\l)}{\l-K} \Bigr\}
\end{equation}
(compare to formula \eqref{GG}). Let us now prove that this is the Green's function of a Jacobi matrix.
We consider two cases:\\
\textbf{Case A)} $K\not\in \s\cup\tilde\s$ \\

Expanding $\frac{\hat N(\l)}{\l-K}$ in partial fractions we get:
\begin{equation}\label{fractions}
\hat G(\l) = \frac{1}{1-\theta^2} \Bigl\{\frac{\theta^2-\beta _0}{\l-K} - \sum_{j=1}^N\frac{\beta _j}{\lambda -\lambda _j} \Bigr\}
\end{equation}

where
\begin{equation}\label{beta}
\beta _j= \lim_{\lambda \to\lambda_j}\frac{\hat N(\l)}{\l-K} (\lambda -\lambda_j)=\frac{\prod_{i=1}^N (\lambda_j-\hat{\lambda}_i)}{\prod_{i\not=j}^N(\lambda_j-{\lambda}_i)(\lambda_j -K)} \quad\mbox{if}\quad j\not=0
\end{equation} and

\begin{equation}\label{beta1}
\beta _0=\hat N(K)
\end{equation}
 From condition III) and \eqref{beta1} we obtain
 \begin{equation}\label{hatGG}
\hat G(\l) = \frac{1}{1-\theta^2}\sum_{j=1}^N \frac{\beta _j}{\lambda_j -\lambda }
\end{equation}
 Now, from \eqref{hat G} and \eqref{hatGG} we find that
\begin{equation}\label{uno}
\lim_{\lambda \to\infty}\lambda \hat G(\l)=-1 \quad\mbox{and}\quad \lim_{\lambda \to\infty}\lambda \hat G(\l)
  =-\frac{1}{1-\theta^2}\sum_{i=1}^N \beta _i
  \end{equation}
respectively. Therefore
\begin{equation}\label{dos}
\frac{1}{1-\theta^2}\sum_{i=1}^N \beta _i =1
\end{equation}
and
\begin{equation}\label{hatGo}
\hat G(\lambda )=\sum_{i=1}^N\frac{\alpha  _i}{\lambda_i -\lambda } \quad \mbox{with}\quad \sum_{i=1}^N\alpha _i=1
\end{equation}
where $\alpha_i:=\frac{\beta _i}{1-\theta ^2} $.\\
From \eqref{beta} we know that $\beta _j=0$ if and only if $\lambda _j\in \sigma \cap \hat\sigma $. From condition III) it follows that
the sum in \eqref{hatGo} has $k:=N-q$ terms and $ N-\tilde n\leq k\leq N$.  Using the expression \eqref{beta} and the interlacing condition I), it follows
that if  $\alpha _i\not=0$ then $\alpha _i>0$.
Therefore

\begin{equation}\label{hatGoo}
\hat G(\lambda )=\sum_{l=1}^k\frac{\alpha  _{i_l}}{\lambda_{i_l} -\lambda } \quad \mbox{with}\quad \sum_{l=1}^k\alpha _{i_l}=1 \quad \mbox{and}\quad
\alpha _{i_l}>0
\end{equation}
 According to theorem 6.2 of \cite{MR1616422}, \eqref{hatGoo} implies that  $\hat G(\lambda )$ has the form of a $nn$ Green's function for at least one Jacobi matrix $J$.
Now we shall describe the family of matrices which have $\hat G(\lambda )$ as its $nn$ Green's function and moreover have spectrum equal to the given
set $\sigma $.

All finite Jacobi operators with a  $nn$ Green's function given by  $\hat G(\lambda )$ in \eqref{hatGoo} have the same eigenvalues $\lambda_{i_l} , l=1,...,k$,
but the other $N-k$ eigenvalues may change.
To study the family of operators which correspond to a given Green'function, we
shall use the theory of interior inverse problems for finite Jacobi matrices, developed in \cite{MR1616422}, theorems 6.1-4.
The key formula of this method is (2.18) in \cite{MR1616422}:
\begin{equation}\label{GmmGesztSim}
 -\langle \delta _n, (J-z)^{-1} \delta _n\rangle^{-1} = \l - a_n  + b_n^2 m_+(\l,n) + b_{n-1}^2 m_-(\l,n),
\end{equation}
where $$m_+(\l,n):= \langle \delta _{n+1}, (J_{[n+1,N-1]}-z)^{-1} \delta _{n+1}\rangle$$ $$m_-(\l,n):= \langle \delta _{n-1}, (J_{[0,n-1]}-z)^{-1} \delta _{n-1}\rangle$$ are the so called
m-Weyl functions. The matrices  $J_{[n+1,N-1]}$ and  $J_{[0,n-1]}$ were defined in Lemma \ref{Green}.
 It happens that $m_+(\l,n)$ determines $J_{[n+1,N-1]}$ uniquely (see Remark \ref{exp}),  and has the form
 \begin{equation}\label{m+}
m_+(\l,n)=\sum_{i=1}^{N-1-n}\frac{\gamma _i}{f_i-\lambda }, \quad \gamma _i>0
 \end{equation}
where $\sum_{i=1}^{N-1-n}\gamma _i=1$ and the $f_i$ are the eigenvalues of $J_{[n+1,N-1]}$. Any sum of this form is legal for $m_+(\l,n)$.
 Similarly  $m_-(\l,n)$ determines uniquely $J_{[0,n-1]}$ and has the form
 \begin{equation}\label{m-}
m_-(\l,n)=\sum_{i=1}^{n}\frac{\kappa _i}{g_i-\lambda }, \quad \kappa  _i>0
 \end{equation}
  where  $\sum_{i=1}^{n}\kappa  _i=1$ and $g_i$ are the eigenvalues of $J_{[0,n-1]}$. Any such sum is allowed for $m_-(\l,n)$.

 The reconstruction procedure is as follows:\\
 Given $\hat G(\lambda )$ as in \eqref{hatGoo} then
   \begin{equation}\label{rec}
  -\hat G(\lambda )^{-1}=z-a +\sum_{l=1}^{k-1}\frac{\beta _l}{\nu  _l-\lambda }
   \end{equation}
where $\nu  _1<\nu_2< ...<\nu _{k-1}$ are the zeros of  $\hat G(\lambda )$. The numbers $\nu _l,a $ and $\beta _l>0$ are determined by
$\alpha _{i_l}$ and $\lambda _{i_l}$ in the expression \eqref{hatGoo}. Now we have to write the right side of equality \eqref{rec} in the form
of the right side of equality  \eqref{GmmGesztSim} for some $a_n, b_n^2, b_{n-1}^2$ and $m_+(\l,n),m_-(\l,n)$ of the form described in
\eqref{m+} and \eqref{m-}. If we do this, then we would have according to \eqref{GmmGesztSim} that $\hat G(\lambda )$ is a $nn$ Green's function
for a matrix $J$ with corresponding entries  $a_n, b_n, b_{n-1}$ and submatrices $J_{[n+1,N-1]}$ and $J_{[0,n-1]}$ determined by the m-Weyl functions.
 From condition II) we know that the $q$ points of $\sigma \cap \tilde\sigma$ are among  zeros of $\hat G(\lambda )$. We will construct $m_+(\l,n)$ and $m_-(\l,n)$ in such a way
these $q$ points are common poles of them. The other $k-1-q$ zeros of $\hat G(\lambda )$ will be poles of just one of the m-Weyl functions.
Since $ m_-(\l,n)$ has $n$ poles, then there are
\begin{equation}\label{choice}
 {k-1-q \choose n-q}={N-2q-1 \choose n-q}
\end{equation}
 possibilities of distributing non common poles.
For each of the $q$ common poles $\mu _l$ we pick a decomposition  $\beta _l= \beta _l^{(1)}+\beta _l^{(2)}$such that the addends  $\frac{\beta _l^{(i)}}{\mu _l-\lambda }, i=1,2$
appear each in one of the sums \eqref{m+}, \eqref{m-}. So we have $q$ parameters which generate a manifold for each one of the possible choices \eqref{choice}.
That these manifolds are diffeomorphic to a sphere follows from theorem 3.6 of \cite {MR1616422}.
 Since we have constructed the m-Weyl functions, the matrices  $J_{[n+1,N-1]}$ and  $J_{[0,n-1]}$ are determined. We can fix
 $a=a_n$ and $$b_n^2= \sum_{ l \,so \,\mu _l \, is\, an f_i\atop \mu _l \, is \,not\, an\, g_i} \beta _l
+\sum_{ l \,so \,\mu _l \, is\, an f_i\atop and\, an\, g_i} \beta _l ^{(1)}
\quad \mbox{and}\quad b_{n-1}^2=\sum_{ l \,so \,\mu _l \, is\, an \,g_i\atop\mu _l \, is \,not\, an\, f_i} \beta _l+\sum_{ l \,so \,\mu _l \, is\, an f_i\atop and\, an\, g_i} \beta _l ^{(2)} $$ where $f_i$  and $g_i$ are defined in \eqref{m+}, \eqref{m-}. So given $\hat G(\lambda )$ we have constructed
 a family of Jacobi matrices so that each member of it, has $\hat G(\lambda )$ as its $nn$ Green's function and its eigenvalues are exactly the points of
 $\sigma $. We have only to prove that if we perturb one of these matrices $J$, then the perturbation $\tilde J$ has spectrum exactly $\hat\sigma$.
Consider then the Jacobi matrix $\tilde J$, obtained from $J$
by formulas \eqref{J pert} with  $M:=(\theta^{-2}-1) K$, where $K$ and $\theta ^2$ are as
defined above in theorems \ref {CondNecSuf}. Then, by formulas
\eqref{GG} and \eqref{hat G},
$$
\frac{1}{1-\theta^2} \Bigl\{ \frac{\theta^2}{\l-K} - \frac{\hat N(\l)}{\l-K} \Bigr\}
= \frac{1}{1-\theta^2} \Bigl\{ \frac{\theta^2}{\l-K} - \frac{N(\l)}{\l-K} \Bigr\}
$$
which implies $\hat N(\l)=N(\l)$ and $\prod(\l-\tilde\l_j)\equiv\prod(\l-\hat\l_j)$, thus
$$
\hat\s=\{\hat \l_1,\ldots,\hat \l_N\}=\{\tilde\l_1,\ldots,\tilde\l_N\}=\s(\tilde J).$$
 This proves all assertions of theorems \ref {CondNecSuf} and \ref{manifolds} for case A)$K\not\in \s\cup\tilde\s$.\\

\textbf{  Case B)}$K\in\sigma \cup \hat \sigma$.\\
In this case from condition IV) we know that there exist $j_0\in \{1,...,N\}$ such that $K=\lambda _{j_0}=\hat\lambda _{j_0}$.
Then \eqref{fractions} takes the form:  \\

\begin{equation}\label{fract}
\hat G(\l) = \frac{1}{1-\theta^2} \Bigl\{\frac{\theta^2-\beta {j_0}}{\l-\lambda _{j_0}} - \sum_{j\not=j_0}^N\frac{\beta _j}{\lambda -\lambda _j} \Bigr\}
\end{equation}
 where
 \begin{equation}\label{betaa}
\beta _j=\frac{\prod_{i\not=j_0}^N (\lambda_j-\hat{\lambda}_i)}{\prod_{i\not=j}^N(\lambda_j-{\lambda}_i)} \quad\mbox{if}\quad j\not=j_0
\end{equation} and

\begin{equation}\label{betaa1}
\beta _{j_0}=\hat N(K)
\end{equation}
 Analogously to what was done in (\ref{uno}), (\ref{dos}) and \eqref{hatGo} we get
 \begin{equation}\label{doos}
\frac{1}{1-\theta^2}\sum_{i=1}^N \beta _i -\frac{\theta^2}{1-\theta^2} =1
\end{equation}
and
\begin{equation}\label{hatGor}
\hat G(\lambda )=\sum_{i=1}^N\frac{\alpha  _i}{\lambda_i -\lambda } \quad \mbox{with}\quad \sum_{i=1}^N\alpha _i=1
\end{equation}
where $\alpha_i:=\frac{\beta _i}{1-\theta ^2} $ if $i\not=j_0$ ,\quad $\alpha _{j_0}=\frac{\beta _{j_0}-\theta ^2}{1-\theta^2}$.
For $i\not=j_0$ we get using \eqref{betaa}  that $\alpha _i=0$ if and only if $\lambda _i=\hat\lambda _i$, that is exactly $q$ times
according to \eqref{mus}.Using the interlacing condition I) and \eqref{betaa} it follows that  $\alpha _i>0$ if  $\alpha _i\not=0$.
 Therefore
 \begin{equation}\label{hatGora}
\hat G(\lambda )=\sum_{l=1}^{k-1}\frac{\alpha  _{i_l}}{\lambda_{i_l} -\lambda }+\frac{\alpha _{j_0}}{\lambda _{j_0}-\lambda }
\quad \mbox{with}\quad \sum_{l=1}^{k-1}\alpha _{i_l} +\alpha _{j_0} =1,\quad \alpha _{i_l}>0
\end{equation}
Where $k:=N-q$. Now we have two options:\\
 i) If situation IV)a happens, then we get exactly \eqref{hatGoo} since $\alpha _{j_0}>0$, $K$ is a pole of $\hat G(\lambda )$ and the analysis is completely analogous to
 Case A).\\
 i))If IV)b holds then $\alpha _{j_0}=0$ , then
 \begin{equation}\label{hatGoraa}
\hat G(\lambda )=\sum_{l=1}^{k-1}\frac{\alpha  _{i_l}}{\lambda_{i_l} -\lambda }
\quad \mbox{with}\quad \alpha _{i_l}>0
\end{equation}
 and $K$ is a zero of $\hat G(\lambda )$. So we have $q+1$ fixed zeros. An analysis similar to the one for Case A) gives
 $${k-2-(q+1) \choose n-(q+1)}={N-2q-3 \choose n-q-1}$$
 possible choices and then $q+1$ parameters.

\end{proof}

\begin{remark}\label{exp}
To find the entries of the matrices $J_{[0,n-1]}$ and $J_{[n+1,N-1]}$ one could use the continuous fraction expansions of the Weyl functions:

\begin{equation} \label{pasita}
 -m_-(\l,n)^{-1}= \lambda - a_{n-1} -
{b_{n-2}^2 \over\displaystyle \lambda - a_{n-2} -
 {\strut b_{n-3}^2 \over\displaystyle \lambda - a_{n-3} - } }
{\atop \displaystyle {\strut
 \rule{0mm}{2mm} \makebox[5mm]{} \atop \ddots
{\atop \displaystyle {\strut
 b_0^2 \over\displaystyle \lambda - a_0 }. } } }
\end{equation}
and
\begin{equation} \label{pasita}
 -m_+(\l,n)^{-1}= \lambda - a_{n+1} -
{b_{n+1}^2 \over\displaystyle \lambda - a_{n+2} -
 {\strut b_{n+2}^2 \over\displaystyle \lambda - a_{n+3} - } }
{\atop \displaystyle {\strut
 \rule{0mm}{2mm} \makebox[5mm]{} \atop \ddots
{\atop \displaystyle {\strut
 b_{N-2}^2 \over\displaystyle \lambda - a_{N-1} }. } } }
\end{equation}

These expansions are unique and can be obtained using Euclid's algorithm.
We have to choose the negative square root of the $b_{i}^2$ since the off diagonal terms
of our Jacobi matrices are negative, see formula \eqref{klm}.

\end{remark}

\begin{acknowledgments}
  We thank Luis Silva and Laura Oropeza for useful comments and stimulating discussions.
  We gratefully acknowledge Luis Silva's help with \LaTeX.
 M. Kudryavtsev thanks IIMAS-UNAM for financial support which made possible several visits to UNAM.
\end{acknowledgments}

\end{document}